\documentclass[11pt,a4paper]{article} 
\usepackage[latin1]{inputenc}
\usepackage[english]{babel}
\usepackage{amsmath,amssymb,amsthm}
\theoremstyle{plain}  
\newtheorem{theorem}{Theorem} 

\newtheorem{lemma}[theorem]{Lemma}

\newtheorem{remark}[theorem]{Remark}

\usepackage{amssymb} 
\usepackage{comment}   

\title{\bf Minimizing GCD sums and applications to non-vanishing of theta functions and to Burgess' inequality}
 
\author{R\'egis de la Bret\`{e}che \and
  Marc Munsch}
  
\newcommand{\Addresses}{{
  \bigskip
  \footnotesize

  \textsc{R\'{e}gis de la Bret\`{e}che, Institut de Math\'ematiques de Jussieu-PRG 
    UMR 7586, Universit\'e Paris Diderot-Paris 7,
Sorbonne Paris Cit\'e, Case 7012, F-75013 Paris, France}\par\nopagebreak
  \textit{E-mail address:} \texttt{regis.de-la-breteche@imj-prg.fr}

  \medskip

   \textsc{Marc Munsch,  Institut f\"{u}r Analysis und Zahlentheorie 
8010 Graz, Steyrergasse 30, Graz, Austria}\par\nopagebreak
  \textit{E-mail address:} \texttt{munsch@math.tugraz.at}

}}

\date{\today}

\begin{document}
\bibliographystyle{alpha}
\maketitle


\begin{abstract}
In recent years the question of maximizing GCD sums regained interest due to its firm link with large values of $L$-functions. In the present paper we initiate\footnote{After a preprint by the second author \cite{energy} was released, the authors worked together  and obtained several improvements as well as other results which are now contained in this new version.} the study of minimizing for positive weights~$w$ of normalized $L^1$- norm the sum $\sum_{m_1 , m_2 \leqslant N} w({m_1})w({m_2})\frac{(m_1,m_2)}{\sqrt{m_1m_2}} $.  We consider as well the intertwined question of minimizing a weighted version of the usual multiplicative energy. We give three applications of our results. Firstly we obtain a logarithmic refinement of Burgess' bound on character sums $\displaystyle{\sum_{M<n\leqslant M+N}\chi(n)}$ improving previous results of Kerr, Shparlinski and Yau. Secondly let us denote by $\theta (x,\chi)$ the theta series associated to a Dirichlet character $\chi$ modulo $p$. Constructing a suitable mollifier, we improve a result of Louboutin and the second author and show that,  for any $x>0$, there exists at least $ \gg p/(\log p)^{ \delta+o(1)}$ (with $\delta=1-\frac{1+\log_2 2}{\log 2} \approx 0.08607$) even characters such that $\theta(x,\chi) \neq 0$. Lastly, we obtain lower bounds on small moments of character sums.
\end{abstract}

\bibliographystyle{plain}
\maketitle

\section{Introduction}\label{Intro}
\subsection{GCD sums}\label{gcd}
\hspace{\parindent} 
Let $S_\alpha({\cal M})$ be the G\'al's sum associated with a set ${\cal M}$ and  defined by 
$$S_\alpha( {\cal M}):=\sum_{m_1,m_2\in{\cal M}} \frac{(m_1,m_2)^{2\alpha}}{ (m_1 m_2)^\alpha}\qquad (0<\alpha\leqslant 1) $$ where as usual $(m_1,m_2)$ denotes the greatest common divisor of $m_1$ and~$m_2$. 
Bounding these sums had originally interesting applications in metric Diophantine approximation (see \cite{DuffinHarman, metric}). Recently, further study was carried out due to the connection with large values of the Riemann zeta function (see for instance \cite{polydisc, hilb,GalRad,sound}). In \cite{bs1}, \cite{Gal}, they were used to prove a lower bound of $\max_{t\in [0,T]} |\zeta(\tfrac 12 +it)|$ and $\max_{\chi\in X_p^+} |L(\tfrac 12,\chi)|$ where 
$X_p^+$ is the set of even characters modulo $p$ and $L(s,\chi)$ is the $L$-Dirichlet series associated to a character $\chi$. In \cite{Gal}, La Bret\`eche and Tenenbaum proved that  
$$\max_{|{\cal M}|=N} \frac{S_{1/2}( {\cal M})}{|{\cal M}|}=
\exp\Bigg\{(2\sqrt{2}+o(1))\sqrt{\frac{\log N \log_3N}{ \log_2 N}}\Bigg\},$$
where $\log_k$ is the $k$th-iterative of the logarithm.
In this result, the cardinality of ${\cal M}$ is fixed while the size of its elements is not.
Moreover this estimate was also satisfied by 
$$Q({\cal M}):=\sup_{{\scriptstyle w\in \mathbb{C}^N\atop\scriptstyle  ||w||_2=1}}
\Bigg| \sum_{m_1,m_2\in{\cal M}}w(m_1)\overline{w( m_2)} 
\frac{(m_1,m_2) }{ \sqrt{m_1m_2}}\Bigg| $$
where $||w||_p$ denotes the $p$-norm of the $N$-tuple $w\in \mathbb{C}^N$.

In this article, we study  the minimal value of the ratio
\begin{equation}\label{defT0}
\mathcal{T}_0(N):=\inf_{{ w\in (\mathbb{R}_{+})^N  }}\left( \frac{N}{||w||_1^2}\sum_{m_1 , m_2 \leqslant N} w({m_1})w({m_2})\frac{(m_1,m_2)}{m_1+m_2} \right),
\end{equation}
and
\begin{equation}\label{defT1}
\mathcal{T}_1(N):=\inf_{{ w\in (\mathbb{R}_{+})^N  }}\left( \frac{N}{||w||_1^2} \sum_{m_1 , m_2 \leqslant N} w({m_1})w({m_2})\frac{(m_1,m_2)}{\sqrt{m_1m_2}} \right).
\end{equation}
 Our main result is the following.

\begin{theorem}\label{gcdsumth}
 There exists $ \delta_0 \approx 0.16656< 1/6$ such that, when $N$ tends to~$+\infty$, we have  $$(\log N)^{ \delta_0+o(1)}\leqslant  \mathcal{T}_0(N)\leqslant \mathcal{T}_1(N)\ll   (\log N)^{ \delta_0+o(1)}.$$
\end{theorem} 

 We show that this minimization question arises naturally in three different problems. The first application involves logarithmic improvements of the famous Burgess' bound on multiplicative character sums while the second application is concerned with non-vanishing of theta functions. Though we show in the latter case (cf. Section \ref{multiplisec} and \ref{secburgess}) that a related minimization problem gives better results. As an application of this second minimization problem, we obtain lower bounds on small moments of character sums. We believe that this minimization problem might also have applications in metric Diophantine approximation. 
 
 \subsection{Multiplicative energy}\label{multiplisec}
 For two sets $\mathcal{A},\mathcal{B}\subset \left[1,N\right]$, let us consider the multiplicative energy (as defined for instance in \cite{Gowers4,Taogroup,TaoVu}) 
 $$ E_{\times}(\mathcal{A},\mathcal{B}):=\vert \left\{  m_1,n_1\in \mathcal{A}, m_2,n_2\in \mathcal{B}\,:\, m_1m_2=n_1n_2\right\}\vert.$$  This quantity appears to be of great importance in additive combinatorics. Under the additional restriction $w(m)\in \left\{0,1\right\}$, the weights $w$ introduced in Section \ref{gcd} can be viewed as the characteristic indicator $w=1_{\mathcal{B}}$ of a set~$\mathcal{B}\subset \left[1,N\right]$ of integers. In this setting, the problem of minimizing $\mathcal{T}_0(N)$ amounts to minimize the quantity $S(\mathcal{B})/|\mathcal{B}|^2$ where 
\begin{equation}\label{defS} S(\mathcal{B}):=\sum_{m_1,m_2 \in \mathcal{B} }\frac{(m_1,m_2)}{m_1+m_2}.\end{equation} It is not hard to see that this sum is intimately connected to the quantity
 $$ E_{\times}(N,\mathcal{B}):=\vert \left\{1 \leqslant m_1,n_1\leqslant N, m_2,n_2\in \mathcal{B}\,:\, m_1m_2=n_1n_2\right\}\vert.$$  In view of our applications, we need to bound the multiplicative energy in a symmetric situation, namely $E_{\times}(\mathcal{B},\mathcal{B})$. To be consistent with our previous problem and to give us some flexibility, we define the weighted version of the multiplicative energy:
\begin{equation}\label{weightedenergy} \mathcal{E}(N,w):=\sum_{\scriptstyle m_1,m_2,n_1,n_2 \leqslant N \atop\scriptstyle  m_1m_2=n_1n_2} w(m_1)w(m_2)w(n_1)w(n_2).  \end{equation} We want to minimize this quantity among choices of positive weights and introduce for this purpose
\begin{equation}\label{defm}  \mathcal{E}(N):=\inf_{{\scriptstyle w\in (\mathbb{R}_{+})^N}} \frac{N^2 \mathcal{E}(N,w)}{||w||_1^4}.  \end{equation} When $w=1_{\mathcal{B}}$ is the characteristic indicator of a set $\mathcal{B}$, this equals to minimize $ {N^2 E_{\times}(\mathcal{B},\mathcal{B})}/{\vert \mathcal{B}\vert^4}$. Using similar techniques as in the proof of Theorem~\ref{gcdsumth}, we prove
the following asymptotical result.
\begin{theorem}\label{multh}
Let $\delta:=1-\frac{1+\log_2 2}{\log 2} \approx 0.08607$.  When $N$ tends to~$+\infty$, we have   $$ \mathcal{E}(N) = (\log N)^{ \delta+o(1)}.$$
\end{theorem}  We observe that the exponent $\delta$ is the one appearing in the famous multiplication table problem of Erd\H{o}s \cite{Tenintervalle,Ford}. We did not try to give an explicit estimate of $o(1)$ appearing in the estimate of $\mathcal{E}(N)$.
  
  \subsection{First application: Improvement of Burgess' bound}\label{secburgess}
\hspace{\parindent} 
Let us consider
$
S_\chi(M,N) := \sum_{M<n \leqslant M+N} \chi(n),$
where $\chi \mod p$ is a multiplicative character. The classical bound of P\'olya and Vinogradov gives 
\begin{equation}\label{PV}
|S_\chi(M,N)| \ll \sqrt{p} \log p
\end{equation}
for any non principal $\chi \mod p$. In particular, this is a non trivial result for~$N > p^{1/2+\varepsilon}$.
A major breakthrough was obtained by Burgess \cite{Burgesssaving,Burgesssaving2} implying a saving for intervals of length $N \geqslant p^{1/4+\varepsilon}$. Precisely, for any prime number $p$, non trivial multiplicative character modulo $p$ and integer $r\geqslant 1$, Burgess proved the following inequality
\begin{equation} \label{eq:Burgess}
\vert S_{\chi}(M,N)\vert  \ll N^{1-1/r} p^{(r+1)/4r^2} \log p,
\end{equation} where the constant depends only on $r$. Even though much stronger results are expected, this bound remains nowadays the sharpest that could be obtained unconditionally.
 
  However, some logarithmic refinements were obtained unconditionally (see \cite[Chapter 14]{IK} following ideas from \cite{FrI}). The last in date is due to Kerr, Shparlinski and Yau who proved for $r\geqslant 2$
\begin{equation}\label{IgorKam} \vert S_{\chi}(M,N) \vert \ll N^{1-1/r} p^{(r+1)/4r^2} (\log p)^{1/4r}. \end{equation} These improvements rely on an averaging argument which leads to count the number of solutions of certain congruences modulo $p$. Initially, the averaging was carried out over the full interval while in \cite{Burgessrefine}, the authors restricted it over numbers without small prime factors. Theorem~\ref{gcdsumth} allows us to perform a similar argument with a set of higher density than the one considered in~\cite{Burgessrefine}. We use this in order to prove the following result.
\begin{theorem}
\label{Burgess}
Let $p$ be prime, $r\geqslant 2$, $M$ and $N$ integers with 
$$N\leqslant p^{1/2+1/4r}.$$
For any nontrivial multiplicative character $\chi$  modulo $p$, 
\begin{align*}
\vert S_{\chi}(M,N)\vert  &\ll N^{1-1/r}p^{(r+1)/4r^2} \max_{1\leqslant x\leqslant p} \mathcal{T}_0(x)^{1/2r}\cr&\ll N^{1-1/r}p^{(r+1)/4r^2} (\log p)^{(\delta_0+o(1))/2r}
\end{align*}  where $\delta_0 \approx 0.16656$ as in Theorem \ref{gcdsumth}.
\end{theorem}

  \goodbreak

\subsection{Second application: Non vanishing of theta functions}
  \hspace{\parindent}  The distribution of values of $L$-functions is a deep question in number theory which has various important repercussions for the related attached arithmetic, algebraic and geometric objects. The main reason comes from the fact that these values and particularly the central ones hold a lot of fundamental arithmetical information, as illustrated for example by the famous Birch and Swinnerton-Dyer Conjecture \cite{BSD1,BSD2}. It is widely believed that they should not vanish unless there is an underlying arithmetic reason forcing it. Consider the Dirichlet $L$-functions associated to Dirichlet characters $$L(s,\chi):=\sum_{n \geqslant 1} \frac{\chi(n)}{n^s} \hspace{5mm} (\Re e(s) > 1).$$ In this case there exists no algebraic reason forcing the $L$-function to vanish at $s=\tfrac{1}{2}$. Therefore it is certainly expected that $L(\frac{1}{2},\chi) \neq 0$ as firstly conjectured by Chowla \cite{Chowla} for quadratic characters. In the last century the notion of family of $L$-functions has been important as heuristic guide to understand or guess many important statistical properties of $L$-functions. One of the main analytic tools is the study of moments and various authors have obtained results on the mean value of these $L$-series at their central point $s=\tfrac{1}{2}$.

 Using the method of mollification, it was first proved by Balasubramanian and Murty \cite{BalaMurty} that there exists a positive proportion of characters such that the $L$-function does not vanish at $s=\tfrac{1}{2}$. Their result was improved and greatly simplified by Iwaniec and Sarnak \cite{ISoriginal} enabling them to derive similar results for families of automorphic $L$-functions \cite{ISauto}. Since then, a lot of technical improvements and generalizations have been carried out, see for instance \cite{Bui,Khan,Soundquad}.      

As initiated in previous works \cite{Debrecen,LM,thetalow,thetaupp}, we would like to obtain similar results for moments 
of theta functions $\theta (x,\chi )$ associated with Dirichlet $L$-functions 
 and defined by
$$ \theta (x,\chi)
=\sum_{n\geqslant 1} \chi (n){\rm e}^{-\pi n^2x/p}\qquad 
\ \ \ (\chi\in X_p^+),$$ 
where 
$X_p^+$ denotes the subgroup of order $\tfrac{1}{2}(p-1)$ of the even Dirichlet characters mod $p$.
  It was conjectured in \cite{efficient} that $\theta(1,\chi) \neq 0$ for every non-trivial character modulo a prime \footnote{Pascal Molin informed the authors that he performed some computations proving that $\theta(1,\chi) \neq 0$ for $p \leqslant 10^6$.} (see \cite{Zagier} for a case of vanishing in the composite case). Using the computation of the first two moments of these theta functions at the central point $x=1$, Louboutin and the second author \cite{LM} obtained that $\theta (1,\chi)\neq 0$ for at least $p/\log p$ even characters modulo $p$ (for odd characters,  
a similar result was already proven by Louboutin in \cite{CRAS}). Constructing different kind of mollifiers than in the case of $L$-functions, we get the following improvement. 

\begin{theorem}\label{mainth}
Let $x>0$. For all sufficiently large prime $p$, there exists at least $$\gg  \frac{p}{\mathcal{E}(\sqrt{ {p}/{3}} )} \gg \frac{p}{(\log p)^{ \delta+o(1)}}$$ even characters $\chi$ such that $\theta(x,\chi) \neq 0$, where $\delta=1-\frac{1+\log_2 2}{\log 2} \approx 0.08607$ as before.
\end{theorem}

\subsection{Third application: Lower bounds on small moments of character sums}
\hspace{\parindent}  As in Section \ref{secburgess}, we consider $S_{\chi}(N)= \sum_{m\leqslant N} \chi(n)$ where $\chi$ is a multiplicative character modulo a prime $p$. Using probabilistic techniques, Harper \cite{Harpelow} recently proved Helson's conjecture about the first moment of Steinhaus random multiplicative functions (multiplicative random variables whose values at prime integers are uniformly distributed on the complex unit circle). He also investigated the deterministic case and obtained upper bounds on the first moment of character sums. Obtaining sharp lower bounds from the probabilistic methods used in \cite{Harpelow} seems harder \footnote{Private communication with Adam Harper.}. Using Theorem \ref{multh}, we obtain the following lower bound on the $L^r$- norm of character sums. 

\begin{theorem}\label{firstmomentcarac}
Let us fix $4/3<r<2$. For $N\geqslant 1$ and $p$ sufficiently large and $N<\sqrt{p}$, we have  
\begin{equation}
\frac{1}{p-1} \sum_{\chi \neq \chi_0} \left\vert S_{\chi}(N)\right\vert^r \gg \frac{N^{r/2}}{\mathcal{E}(N)^{1-r/2}}.
 \end{equation}In particular, for $p$ sufficiently large and $N<\sqrt{p}$, we have
$$
\frac{1}{p-1} \sum_{\chi \neq \chi_0} \left\vert S_{\chi}(N)\right\vert \gg \frac{\sqrt{N}}{\mathcal{E}(N)^{1/2}} \gg \frac{\sqrt{N}}{(\log N)^{ \delta /2+o(1)}}$$ with $\delta/2 \approx 0.043$ and $\delta$ defined in Theorem \ref{multh}.
\end{theorem}
\begin{remark}
This result can be easily generalized to
composite moduli, but for the sake of simplicity and coherence, we restricted the presentation to the case of prime moduli.
\end{remark}

The same method can be also applied to get a lower bound for 
$$ \frac{1}{T}\int_0^T \Bigg|\sum_{n\leqslant N}n^{it}\Bigg|^r {\rm d} t.$$ The study of the limit $\displaystyle{\lim_{T\to+\infty} \frac{1}{T}\int_0^T  \Bigg|\sum_{n\leqslant N}n^{it}\Bigg|^r  {\rm d} t}$ was initiated by Helson \cite{H06} and further investigated by Bondarenko and Seip in~\cite{BS16}. For any $r \leqslant 1 $, they proved a lower bound of size $\sqrt{N} (\log N)^{-0.07672}$ and obtained for $r=1$  the same bound with an exponent $-0.05616$ using a different method than ours. Their method relies on  \cite[Lemma $3$]{BS16} which does not exist for character sums. We illustrate Theorem \ref{multh} by the following estimates.

\begin{theorem}\label{firstmomentpolyzeta}
Let us fix $4/3<r<2$. For $N\geqslant 1$ and  $N^2<T$, we have  
\begin{equation} 
 \frac{1}{T}\int_0^T  \Bigg|\sum_{n\leqslant N}n^{it}\Bigg|^r {\rm d} t\gg \frac{N^{r/2}}{\mathcal{E}(N)^{1-r/2}}.
 \end{equation}In particular,  we have
$$
\lim_{T\to+\infty} \frac{1}{T}\int_0^T  \Bigg|\sum_{n\leqslant N}n^{it}\Bigg|  {\rm d} t \gg \frac{\sqrt{N}}{\mathcal{E}(N)^{1/2}} \gg \frac{\sqrt{N}}{(\log N)^{ \delta /2+o(1)}}$$ with $\delta/2 \approx 0.043$ and $\delta$ defined in Theorem \ref{multh}.
\end{theorem}

\begin{remark}
Studying the proofs of \cite{BS16}, it is not difficult to see that their method gives also the same exponent $\tfrac\delta2$ \footnote{This was pointed out by ``Lucia" without detailing the proof on mathoverflow https://mathoverflow.net/questions/129264/short-character-sums-averaged-on-the-character in May of 2017. As it was quoted by ``Lucia", the method of \cite{BS16} relies on some input from analysis (lemma 3 of \cite{BS16})
which permits to restrict the sum over the set of integers $n$ such that $\Omega(n)$ is constant whereas we avoid this part using some weights.}. As the proof of our Theorem \ref{firstmomentpolyzeta} is similar to Theorem \ref{firstmomentcarac}, we do not give any details.\end{remark}

\section{Proof of Theorem \ref{gcdsumth}}
For any sequence $w$, we prove an upper bound on $\mathcal{T}_0(N,w)$ defined by
\begin{equation}\label{defT0preuve}
\mathcal{T}_0(N,w):=  \frac{N}{||w||_1^2 }\sum_{m_1 , m_2 \leqslant N} w({m_1})w({m_2})\frac{(m_1,m_2)}{m_1+m_2}   .
\end{equation} As the proof works for $\mathcal{T}_1(N)$, to get an upper bound, we study 
\begin{equation}\label{defT1preuve}
\mathcal{T}_1(N,w):=  \frac{N}{||w||_1^2 }\sum_{m_1 , m_2 \leqslant N} w({m_1})w({m_2})\frac{(m_1,m_2)}{\sqrt{m_1m_2}}   .
\end{equation}
 We consider the case where the weights are supported on the set of integers with a fixed number of prime  factors.   
Precisely, we choose 
\begin{equation}\label{defwk}
w(m)=w_k(m):= 
\begin{cases}
   1  & \text{if } \Omega(m)=k, \\
   0        & \text{otherwise.}
  \end{cases}
\end{equation} where $k=\kappa\log_2N\in \mathbb{N}$,   $\kappa\in ]0,1[$,  and 
$\Omega(n)$  denotes the number of prime factors of $n$ counted with multiplicity.   
 We introduce the function~$Q$ defined by
$$Q(\lambda):=\lambda\log\lambda-\lambda+1.$$
It is decreasing in the range $[0,1]$ and increasing in $[1,+\infty[$. Assuming  $\kappa\in [\kappa_0,2-\kappa_0]$ with $\kappa_0$ fixed in $ ]0,1[$, we have uniformly (see for instance \cite[Chapter II.6, Theorem $6.5$]{Tencourse}) for large $N$
\begin{equation}\label{asymp} ||w_k||_1=\sum_{m\leqslant N}w_k(m)\asymp \frac{N}{ (\log N)^{Q(\kappa)}\sqrt{\log_2N}}.\end{equation} 
 Moreover, when $k=\kappa\log_2N\in \mathbb{N}$ and  $\kappa\in [0,2-\kappa_0]$, we have uniformly
\begin{equation}\label{uniupper}
\sum_{m\leqslant N}w_k(m)\ll \frac{N}{ (\log N)^{Q(\kappa)} }.\end{equation}
But we have also an uniform upper bound without restriction on $k$
\begin{equation}\label{uniupper*}
\sum_{m\leqslant N}w_k(m)\ll \frac{N}{ (\log N)^{\min\{Q(\kappa)  ,3/8\}} },\end{equation}
since $\tfrac38<Q(2)$.\goodbreak

In order to bound $\mathcal{T}_1(N,w_k)$,
we write $$\mathcal{T}_1(N,w_k)\leqslant 2\frac{N}{||w||_1^2 }\big(S_{1}+ S_{2}\big)$$
with
\begin{align*} S_1 &:= \sum_{d\leqslant \sqrt{N}} \sum_{m_1\leqslant m_2\leqslant N/d} \frac{w_k(dm_1)w_k(dm_2)}{\sqrt{m_1m_2}},\nonumber\\
S_{2} &:= \sum_{m_1\leqslant m_2\leqslant \sqrt{N}} \sum_{d\leqslant N/m_2} \frac{w_k(dm_1)w_k(dm_2)}{\sqrt{m_1m_2}}.
\nonumber
\end{align*}

Let $S_{1}(j)$ and $S_{2}(j)$ be the  contribution in each sum corresponding with~$d$ such that  $\Omega(d)=j=\lambda\log_2N\leqslant k$. We have, using \eqref{uniupper*} once and \eqref{uniupper} twice
\begin{align*}  S_{1}(j)&=\sum_{d\leqslant \sqrt{N}}w_j(d) \sum_{m_1\leqslant m_2\leqslant N/d} \frac{w_{k-j}(m_1)w_{k-j}(m_2)}{ \sqrt{m_1m_2}}\\
&\ll \sum_{d\leqslant \sqrt{N}}w_j(d) \sum_{  m_2\leqslant N/d} \frac{ w_{k-j}(m_2) }{
(\log 2m_2)^{\min\{Q((\kappa-\lambda)\log_2N/\log_2m_2),3/8\}} }
\\
&\ll  N\big(
(\log N )^{-2Q( \kappa-\lambda)  }+(\log N )^{- Q( \kappa-\lambda)-3/8  }\big)\sum_{d\leqslant \sqrt{N}}\frac{w_j(d)  }{ d}
\\
&\ll  N                                              
(\log N )^{-2Q( \kappa-\lambda) -\min\{ 0, 3/8-Q( \kappa-\lambda  )\} } \big(1+
(\log N )^{1-Q( \lambda)  +o(1) }\big)
\\
&\ll  N 
(\log N )^{-2Q( \kappa-\lambda )+1-Q( \lambda)  -\min\{ 0, 3/8-Q( \kappa-\lambda  )\}+o(1) } 
.\end{align*}

 The sums $S_{2}(j)$ satisfy the same kind of bound. 
 We have 
\begin{align*}  S_{2}&(j) =  \sum_{m_1\leqslant m_2\leqslant  \sqrt{N}} \frac{w_{k-j}(m_1)w_{k-j}(m_2)}{ \sqrt{m_1m_2}}\sum_{d\leqslant N/m_2}w_j(d)\\
&\ll  N(\log N)^{-Q(\lambda)}\sum_{m_1\leqslant m_2\leqslant  \sqrt{N}} \frac{w_{k-j}(m_1)w_{k-j}(m_2)}{ m_1^{1/2}m_2^{3/2}}\\
&\ll  N(\log N)^{-Q(\lambda)}\sum_{  m_2\leqslant  \sqrt{N}} \frac{ w_{k-j}(m_2)}{ m_2(\log 2m_2)^{\min\{Q((\kappa-\lambda)\log_2N/\log_2m_2),3/8\}}  } 
\\  
&\ll  N  
(\log N )^{  -Q( \lambda)  }+N
(\log N )^{1-2Q( \kappa-\lambda )-Q( \lambda) -\min\{ 0, 3/8-Q( \kappa-\lambda)\} +o(1) } 
.\end{align*}
 
 Then, integrating over $j$,
$$S_{1}+ S_{2}\ll(\log_2N)\max_{j\leqslant k} (S_{1}(j)+S_{2}(j)) \ll 
N 
(\log N )^{-g_0(\kappa) +o(1) } $$
with 
$$g_0(\kappa):=\min\Big\{ \min_{\lambda\in [0,\kappa]}\big(2Q( \kappa-\lambda ) +Q( \lambda)-1, Q( \kappa-\lambda ) +Q( \lambda)-\tfrac58\big), Q( \kappa)\Big\}.$$

The minimum on $\lambda$ of the first expression is obtained when $\lambda$ is the solution $\leqslant \kappa$ of  $2\log (\kappa-\lambda)=\log  \lambda,$ id est
$$\lambda=\lambda_{\kappa,1}:= \tfrac12\big( 2\kappa+1-\sqrt{4\kappa+1}\big) .$$  Moreover,  the minimum on $\lambda$ of the second term is obtained when $\lambda$ is the solution $\leqslant \kappa$ of  $\log (\kappa-\lambda)=\log  \lambda,$ id est
$$\lambda=\lambda_{\kappa,2}:= \tfrac{1}{2}\kappa .$$

So we have
$$\min\big\{2Q( \kappa-\lambda_{\kappa,1})  +Q( \lambda_{\kappa,1})-1,  2Q( \tfrac12\kappa  ) -\tfrac58, Q( \kappa)\big\} \leqslant g_0(\kappa) $$ and by (\ref{asymp}) we deduce
$$\mathcal{T}_1(N,w_k) \ll  (\log N)^{    f_0(\kappa)+o(1) }$$
with 
\begin{align*}
f_0(\kappa)&:= \max\big\{1+2Q(\kappa)-2Q( \kappa-\lambda_{\kappa,1}) -Q( \lambda_{\kappa,1}),  2Q(\kappa)-2Q( \tfrac12\kappa) + \tfrac58, Q(\kappa)\big\}.
\end{align*}
It remains to choose $\kappa$ to minimize $f_0(\kappa)$. 
We choose $\kappa^*$ verifying 
\begin{equation}1+ Q(\kappa^*)-2Q( \kappa^*-\lambda_{\kappa^*,1}) -Q( \lambda_{\kappa^*,1})= 0.
\label{defkappa*}\end{equation}
 In this case, $$\delta_0= \min_{\scriptstyle\kappa\in ]0,1[}\{  f_0(\kappa)\}=  f_0(\kappa^*)= \max\big\{Q(\kappa^*),2Q(\kappa^*)-2Q( \tfrac12\kappa^* ) + \tfrac58\big\} .$$ Solving numerically this equation, we see that $\kappa^*\approx 0.48154$ and $Q(\kappa^*)\approx 0.16656$. We verify numerically that $2Q(\kappa^*)-2Q( \tfrac12\kappa^*) + \tfrac58 \approx 0.1253$.  This implies that $\mathcal{T}_1(N) \leqslant \mathcal{T}_1(N,w_k) \ll (\log N)^{  \delta_0 +o(1) }$   which concludes the proof of the upper bound.
\\

 \hspace{\parindent} 
For any sequence $w$, we prove a lower bound on $\mathcal{T}_0(N,w)$ defined in~\eqref{defT0preuve}. When $n_1m_1 = n_2m_2$, there exists $n $ such that $n_1=nm_2/(m_1,m_2)$ and $n_2=nm_1/(m_1,m_2)$
so that,   for any choice $w$ of positive weights,
$$ \sum_{\scriptstyle  m_1,m_2,n_1,n_2 \leqslant  N \atop\scriptstyle  n_1m_1 = n_2m_2} w(m_1)w(m_2) \leqslant   N\sum_{m_1,m_2\leqslant N}w(m_1)w(m_2)\frac{ (m_1,m_2)}{m_1+m_2}.$$
Hence we obtain 
\begin{equation}\label{minogcdenergy} \mathcal{T}_0(N,w) \geqslant \frac1{||w||_1^{2}}\sum_{\scriptstyle m_1,m_2,n_1,n_2 \leqslant  N \atop \scriptstyle n_1m_1 = n_2m_2} w(m_1)w(m_2) .\end{equation}
By integration over $k$, there exists $0\leqslant k=\kappa \log_2 N \leqslant  \log_2 N$ such that \begin{equation}\label{cas1}||w.w_k||_1 \gg \frac{||w||_1 }{2(1+\log_2 N)}   \end{equation} or
\begin{equation}\label{cas2} ||w.w_+||_1 \gg \tfrac12||w||_1   \end{equation} 
where 
$$w_+(m):=\sum_{\scriptstyle m\leqslant N\atop\scriptstyle \Omega(m)> \log_2N}1.$$

We first look at the former case. Let $\rho=\rho_\kappa\in [1,\tfrac12(1+\sqrt{5})]$ a parameter that we will choose depending on the value of $\kappa$. When $r=\rho\log_2N$ with $\rho\in [1,\tfrac12(1+\sqrt{5})]$, by \eqref{asymp}, we have the lower bound
$$||w_r||_1=\sum_{\scriptstyle n\leqslant N\atop\scriptstyle  \Omega(n)=r}1\gg \frac{N}{ (\log N)^{Q(\rho)+o(1)}}.$$
 By Cauchy-Schwarz inequality, we get
\begin{align*}\frac{N^2}{(\log N)^{2Q(\rho)+o(1)}}||w.w_k||_1^{2}&\ll\left\{\sum_{  m,n  \leqslant  N } w(m )w_k(m)w_r(n)\right\}^2 \cr & =\left\{\sum_{\ell \leqslant  N^2}  \sum_{\scriptstyle \ell=m n  \atop \scriptstyle  m,n\leqslant N} w(m)w_k(m)w_r(n) \right\}^{2} \cr&\leqslant \left\{\sum_{\scriptstyle m_1,m_2,n_1,n_2 \leqslant  N \atop \scriptstyle n_1m_1 = n_2m_2} w(m_1)w(m_2)\right\}H_{k,r}( N),\end{align*}
where
$$H_{k,r}( N):=\big|\big\{  \ell \leqslant N^2 \, :\, \exists n,m\leqslant N \,\quad \ell=nm,\, \Omega(m)=k ,\,\Omega(n)=r\big\}\big|  . $$ Using \eqref{minogcdenergy} and \eqref{cas1}, we obtain
$$\mathcal{T}_0(N,w)  \gg  \frac{N^2}{(\log N)^{2Q(\rho)+o(1)}H_{k,r}(N)}.$$ 
To bound $H_{k,r}( N)$, we observe that, when $\Omega(m)= \kappa \log_2N$, we have $\Omega(\ell)\geqslant (\rho+\kappa)[\log_2N]$. Then, by \eqref{uniupper} and \eqref{uniupper*}, we get
$$H_{k,r}(N)\ll N^2(\log N)^{-\delta_{1,k}+o(1)}
$$
with 
$$\delta_{1,k}:=\max\{ Q(\kappa)+Q(\rho_\kappa),\min\{Q(\rho_\kappa+\kappa),\tfrac38 \} \} .$$ 
We used here that $ \kappa \leqslant 1$ and $ \rho_\kappa \leqslant \tfrac12(1+\sqrt{5})$ to apply \eqref{uniupper}.
We deduce the lower bound
$$\mathcal{T}_0(N,w) \geqslant (\log N)^{\delta^*_{1,\kappa}+o(1)}$$  where
$$\delta_{1,\kappa}^*:= \max\left\{Q(\kappa)-Q(\rho_\kappa),\min\{Q(\rho_\kappa+\kappa),\tfrac38 \}-2Q(\rho_\kappa)\right\}.$$

Since $\max_{\rho\geqslant 1}(Q(\kappa)-Q(\rho ))=Q(\kappa)-Q(1)=Q(\kappa)$, we have $$\mathcal{T}_0(N,w) \geqslant (\log N)^{\delta_1+o(1)}$$  where
$$\delta_1:=\min_{\kappa \in [0,1]}\max\left\{Q(\kappa),\min\{Q(\rho_{\kappa}+\kappa),\tfrac38 \}-2Q(\rho_{\kappa})\right\}.$$ 
Note that the value $\rho= \frac{1}{2}(1+\sqrt{1+4\kappa})$ maximizes the quantity $Q(\rho+\kappa)-2Q(\rho)$.
 We introduce $\kappa_2$ as the unique solution of the equation 
$$ Q(\tfrac{1}{2}(1+\sqrt{1+4\kappa})+\kappa)-2Q(\tfrac{1}{2}(1+\sqrt{1+4\kappa}))=\tfrac38 -2Q(\tfrac{1}{2}(1+\sqrt{1+4\kappa})).$$ We further define 
$$\rho_\kappa:=\begin{cases}1 &\text{ if } 0 \leqslant \kappa \leqslant \kappa^*\cr
\frac{1}{2}(1+\sqrt{1+4\kappa})  &\text{ if } \kappa^*\leqslant \kappa\leqslant \kappa_2
\cr
1 &\text{ if } \kappa\geqslant \kappa_2.
 \end{cases}$$ 
where $\kappa^*$ is defined by \eqref{defkappa*}. For $\kappa \geqslant \kappa_2 \approx 0.6565$, we have $$\min\{Q(1+\kappa),\tfrac38 \}\geqslant Q(1+\kappa_{2})\approx 0.179154> \delta_0
  .$$  For $\kappa \leqslant \kappa_2$, we have $$Q(\tfrac{1}{2}(1+\sqrt{1+4\kappa})+\kappa)-2Q(\tfrac{1}{2}(1+\sqrt{1+4\kappa})) \leqslant \tfrac38-2Q(\tfrac{1}{2}(1+\sqrt{1+4\kappa}). $$ It follows that $$\mathcal{T}_0(N,w) \geqslant (\log N)^{\min\left\{\delta_2,\delta_0\right\}+o(1)}$$  with
$$\delta_2:=\min_{\kappa \in [0,\kappa_2]}\max\left\{Q(\kappa),Q(\rho_{\kappa}+\kappa)-2Q(\rho_{\kappa})\right\}.$$

The minimum value defined by $\delta_2$ is attained when $\kappa$ is solution of the equation
\begin{equation}\label{kappa1}Q(\kappa)-Q(\rho_{\kappa}+\kappa)+2Q(\rho_{\kappa})=0. \end{equation}  It is not hard to see \footnote{We verified it using a computer algebra system.} that the unique solution to \eqref{kappa1} is $\kappa^* \approx 0.48154$ . Hence we have $\delta_2=\delta_0= Q(\kappa^*)\approx 0.16656$. This concludes the proof in this case.

 In the latter case, we choose $r=[\log_2N]$ and, by the same method, we get 
\begin{align*} N^2 (\log N)^{ o(1)} ||w.w_+||_1^{2}&\ll\left\{\sum_{  m,n  \leqslant  N } w(m )w_+(m)w_r(n)\right\}^2  \cr&\leqslant \left\{\sum_{\scriptstyle m_1,m_2,n_1,n_2 \leqslant  N \atop \scriptstyle n_1m_1 = n_2m_2} w(m_1)w(m_2)\right\}H_{+,r}( N),\end{align*}
where
$$H_{+,r}( N):=\big|\big\{ \ell \leqslant N^2 \, :\, \exists n,m\leqslant N \quad \ell=nm,\, \Omega(m)\geqslant  \log_2N,\Omega(n)=r \big\} \big| . $$
Since $ H_{+,r}( N)\ll N^2(\log N)^{-3/8+o(1)}$ and $\tfrac38>\delta_0$ we obtain by  \eqref{minogcdenergy} and \eqref{cas2}
$$\mathcal{T}_0(N,w) \geqslant (\log N)^{\delta_0+o(1)}.$$

\section{Proof of Theorem \ref{multh}}
\hspace{\parindent} 
First we prove the lower bound on $\mathcal{E}(N)$ defined by \eqref{defm}.  Indeed we have, for any choice of positive weights, by Cauchy-Schwarz 
\begin{align*}\left\{\sum_{a,b \leqslant  N} w(a)w(b)\right\}^2&=\left\{\sum_{n \leqslant  N^2}  \sum_{\scriptstyle n=ab \atop \scriptstyle  a,b \leqslant  N} w(a)w(b) \right\}^{2} \cr&\leqslant  \mathcal{E}(N,w)\sum_{\scriptstyle n \leqslant N^2\atop\scriptstyle  \exists a,b\leqslant N \, n=ab  } 1.\end{align*} Thus we have 
$$\frac{N^2 \mathcal{E}(N,w)}{||w||_1^4} \geqslant (\log N)^{\delta+o(1)}$$  by the known results on the multiplication table of Erd\H{o}s \cite{Tenintervalle,Ford}. \\

We now focus on the proof of the upper bound. Similarly as before, we set $w =w_k $ as defined in \eqref{defwk} where $k=\kappa\log_2N\in \mathbb{N}$,   $\kappa\in ]0,1[$.  We remark that if $m_1n_1=m_2n_2$ then $n_1$ has to be a multiple of $\frac{m_2}{(m_1,m_2)}$ and similarly $n_2$ has to be a multiple of $ \frac{m_1}{(m_1,m_2)}$.
Then we can parametrize the solution of $m_1n_1=m_2n_2$ by
$$m_1=h d_1,\quad m_2=hd_2,\quad  n_1=\ell d_2,\quad  n_2=\ell d_1,$$ with 
$(d_1,d_2)=1$ so that
\begin{equation}\label{calculE} \mathcal{E}(N,w) =\sum_{\scriptstyle m_1,m_2  \leqslant N \atop\scriptstyle (m_1,m_2)=1 } \left(\sum_{h\leqslant N/\max\{ m_1,m_2\}}w(h m_1)w(h m_2)\right)^2.  \end{equation}

  This immediately implies 
$$ \mathcal{E}(N,w) \leqslant 2E_1 + 2E_2$$ where
\begin{align*} E_1 &:= \sum_{h\leqslant \sqrt{N}} \sum_{m_1\leqslant m_2\leqslant N/h} w_k(hm_1)w_k(hm_2)\sum_{\ell  \leqslant N/m_2} w_k(\ell  m_2),\nonumber\\
E_{2} &:= \sum_{m_1\leqslant m_2\leqslant \sqrt{N}} \left(\sum_{h\leqslant N/m_2}w_k(hm_1)w_k(hm_2) \right)^2.
\nonumber
\end{align*}

Let $E_{1}(j)$ and $E_{2}(j)$ be the  contribution in each sum corresponding with~$h$ such that  $\Omega(h)=j=\lambda\log_2N\leqslant k$. 
As for $h\leqslant \sqrt{N}$, we have
\begin{align*}&\sum_{m_1\leqslant m_2\leqslant N/h}   \frac{w_{k-j}(m_1)w_{k-j}(m_2)}{m_2}\cr&\quad\ll\sum_{ m_2\leqslant N/h}  w_{k-j}(m_2) 
(\log 2m_2)^{-\min\{Q((\kappa-\lambda)\log_2N/\log_2m_2),3/8\}} \cr&\quad\ll\frac Nh(\log N )^{-2Q( \kappa-\lambda)-\min\{ 0, 3/8-Q( \kappa-\lambda  )\}  }\end{align*}
we get, using (\ref{uniupper}) 
\begin{align*}  E_{1}(j)&= \sum_{h\leqslant \sqrt{N}}w_j(h) \sum_{m_1\leqslant m_2\leqslant N/h}  w_{k-j}(m_1)w_{k-j}(m_2)\sum_{\ell  \leqslant N/m_2} w_{j}(\ell )\\
&\ll  N\sum_{h\leqslant \sqrt{N}}w_j(h)\sum_{m_1\leqslant m_2\leqslant N/h}  \frac{w_{k-j}(m_1)w_{k-j}(m_2)}{m_2}
(\log 2h)^{-Q(\lambda)}
\\ 
&\ll   N^2(\log N )^{-2Q( \kappa-\lambda)-\min\{ 0, 3/8-Q( \kappa-\lambda  )\}  }\sum_{h\leqslant \sqrt{N}}\frac{w_j(h)  }{ h}
(\log 2h)^{-Q(\lambda)}
\\ 
&\ll  N^2 
(\log N )^{-2Q( \kappa-\lambda ) -\min\{ 0, 3/8-Q( \kappa-\lambda  )\} }\cr&\quad+N^2 
(\log N )^{-2Q( \kappa-\lambda )+1-2Q( \lambda) -\min\{ 0, 3/8-Q( \kappa-\lambda  )\} +o(1) } 
.\end{align*}

 The sums $E_{2}(j)$ satisfy the same kind of bound. We have 
\begin{align*} & E_{2}(j) =  \sum_{m_1\leqslant m_2\leqslant  \sqrt{N}} w_{k-j}(m_1)w_{k-j}(m_2) \sum_{h\leqslant N/m_2}w_j(h)\sum_{\ell  \leqslant N/m_2} w_{j}(\ell )\\
&\ll  N^2(\log N)^{-2Q(\lambda)}\sum_{m_1\leqslant m_2\leqslant  \sqrt{N}} \frac{w_{k-j}(m_1)w_{k-j}(m_2)}{ m_2^{2}} \\
&\ll N^2(\log N)^{-2Q(\lambda)}\sum_{  m_2\leqslant  \sqrt{N}} \frac{ w_{k-j}(m_2)}{ m_2(\log 2m_2)^{ \min\{Q((\kappa-\lambda)\log_2N/\log_2m_2),3/8\}}  } 
\\  
&\ll   N^2  
(\log N )^{  -2Q( \lambda)  }+N^2
(\log N )^{1-2Q( \kappa-\lambda )-2Q( \lambda) -\min\{ 0, 3/8-Q( \kappa-\lambda  )\} +o(1) } 
.\end{align*}
 
 Then, integrating over $j$,
$$E_{1}+ E_{2}\ll(\log_2N)\max_{j\leqslant k} (E_{1}(j)+E_{2}(j)) \ll 
N ^2
(\log N )^{-g(\kappa) +o(1) } $$
with 
$$g(\kappa):=\min\Big\{ \!\min_{\lambda\in [0,\kappa]}\big(2Q( \kappa-\lambda ) +2Q( \lambda)-1, Q( \kappa-\lambda ) +2Q( \lambda) -\tfrac58 \big), 2Q( \kappa)\Big\}.$$ The minimum on $\lambda$ of the first expression is obtained when $\lambda:= \tfrac{1}{2}\kappa $ as before. Moreover, the minimum of the second term is obtained when $\lambda$ is the solution $\leqslant \kappa$ of  $\log (\kappa-\lambda)=2\log  \lambda,$ id est
$$\lambda=\lambda_{\kappa,3}:= \tfrac12\big(\sqrt{4\kappa+1}-1\big).$$
So we have
$$g(\kappa) \geqslant \min\big\{4Q(\tfrac{1}{2}\kappa)-1, Q( \kappa-\lambda_{\kappa,3})+ 2Q(\lambda_{\kappa,3})-\tfrac58, 2Q( \kappa)\big\}.$$
Inserting \eqref{asymp} in \eqref{defm}
$$ \mathcal{E}(N) \ll  (\log N)^{ f(\kappa)+o(1) }$$
with 
\begin{align*}
f(\kappa):=  \max\big\{4Q(\kappa)-4Q(\tfrac{1}{2}\kappa)+1, 
 4Q(\kappa)-Q( \kappa-\lambda_{\kappa,3})- 2Q(\lambda_{\kappa,3})+\tfrac58, 2Q(\kappa)\big\}.
\end{align*}
It remains to choose $\kappa$ to minimize $f(\kappa)$. This occurs for $\kappa^*=1/\log 4$ verifying $$1+ 2Q(\kappa^*)-4Q(\tfrac1{2}{\kappa^*})= 0.$$ In this case, $$\delta= \min_{\scriptstyle\kappa\in ]0,1[}\{  f(\kappa)\}=  f(1/\log 4)= \max\{2Q(1/\log 4), \alpha\} $$  where 
$$ \alpha:= 4Q(1/\log 4)-Q( 1/\log 4-\lambda_{1/\log 4,3})- 2Q(\lambda_{1/\log 4,3})+ \tfrac58 \approx 0.046.$$ This implies that $\mathcal{E}(N) \ll (\log N)^{\delta+o(1)}$ with $\delta= 1-\frac{1+\log_2 2}{\log 2}\approx 0.08607$ which concludes the proof.

\section{Logarithmic improvement of Burgess' bound}

\subsection{Preliminary results}
\hspace{\parindent} 
The following result is a consequence of the Weil bounds for complete character sums, see for instance~\cite[Lemma~12.8]{IK}.
\begin{lemma}
\label{moments}
Let $r \geqslant 2$ be an integer, $B\geqslant 1$,  $p$ a prime and $\chi$ a nontrivial multiplicative character  modulo $p$. Then we have
$$\sum_{u=1}^{p}\left|\sum_{1\leqslant b \leqslant B}\chi(u+b) \right|^{2r}\leqslant  (2r)^r B^{r}p + 2rB^{2r}p^{1/2}.$$
\end{lemma}

For any fixed couple $(a_1,a_2)$, we denote by $T(a_1,a_2;M,N)$  the number of solutions $M<n_1,n_2\leqslant M+N $ of the congruence 
\begin{equation}
\label{congruence}
n_1a_1\equiv n_2a_2 \, (\bmod\, p) 
\end{equation} 
and
$$T_w(M,N,A):=\sum_{a_1,a_2\leqslant A}w(a_1)w(a_2) T(a_1,a_2;M,N)$$
Following the lines of the proof of \cite[Lemma 4.1]{Burgessrefine}, we can prove the following upper bound.

\begin{lemma}\label{abcd}
Let $p$ be a prime and $M,N$,$A$ integers such that 
\begin{equation}\label{bornes} A \leqslant N, \hspace{3mm} AN \leqslant p.  \end{equation}  
For any sequence $w\in \mathbb{C}^N$, we have the following upper bound
$$ T_w(M,N,A) \ll  \Big(\sum_{a \leqslant A}|w(a )|\Big)^2+N\sum_{a_1,a_2\leqslant A}|w(a_1)w(a_2)|\frac{ (a_1,a_2)}{a_1+a_2}  .$$
\end{lemma}

\begin{proof} Assume $T(a_1,a_2;M,N)\neq 0$ and $(n_1',n_2')$ to be one fixed solution.
For any $(n_1,n_2)$ solution of (\ref{congruence}), $(n_1-n_1',n_2-n_2')$ is counted  by $E_{a_1,a_2}(8N^2;p)$ where 
$$E_{a_1,a_2}(n;p):=
\sum_{\scriptstyle n_1^2 + n_2^2 \leqslant n\atop{\scriptstyle a_1n_1 \equiv \ a_2n_2 \bmod p}} 1.$$Taking initial intervals in \cite[Lemma $1$]{Ayyad}, we   
  deduce immediately, when  $(a_1a_2,p)=1$, the bound
\begin{equation}\label{energybound} E_{a_1,a_2}(n;p) \ll 1+\frac{n}{p}+\frac{\sqrt{n} (a_1+a_2)}{p(a_1,a_2)}+\sqrt{n}\frac{(a_1,a_2)}{a_1+a_2}.
\end{equation}  The majorant is dominated by 
$O(1+N(a_1,a_2)/(a_1+a_2))$. Summing over $a_1,a_2\leqslant A$, we get the result.
\end{proof}

\subsection{Proof of Theorem \ref{Burgess}}
 \hspace{\parindent} We keep the notations of \cite{Burgessrefine} and follow closely their argument.   We set
$$ \mathcal{T}_0:=\max_{1\leqslant x\leqslant p} \mathcal{T}_0(x)$$ and proceed by induction on $N$. Our induction hypothesis is the following. 
There exists some constant $c$ such that for any integer $M$ and any integer $K<N$ we have
$$
\left|\sum_{M<n\leqslant M+K}\chi(n)\right|\leqslant c K^{1-1/r}p^{(r+1)/4r^2} \mathcal{T}_0^{1/2 r},
$$
and we want to prove that 
\begin{equation}
\label{induction}
\left|\sum_{M<n\leqslant M+N}\chi(n)\right|\leqslant c N^{1-1/r}p^{(r+1)/4r^2}  \mathcal{T}_0^{1/2 r}.
\end{equation} 
As in \cite{Burgessrefine}, $N<p^{1/4}$ forms the basis of our induction. We define similarly the integers $A$ and $B$ by
\begin{equation*}
A=\left \lfloor \frac{N}{16rp^{1/2r}} \right \rfloor \hspace{3mm} \textrm{   and   }\hspace{3mm} B=\left \lfloor   rp^{1/2r} \right \rfloor.
\end{equation*}

For any integers $1\leqslant a\leqslant A$ and $1\leqslant b\leqslant B$, we have 
\begin{align*}
\sum_{M<n\leqslant M+N}\chi(n)&= \sum_{M<n\leqslant M+N}\chi(n+ab) +\sum_{M-ab<n\leqslant M}\chi(n+ab)\cr & \quad -\sum_{M+N-ab<n\leqslant M+N}\chi(n+ab).
\end{align*}
By our induction hypothesis, we have 
$$
\left|\sum_{M-ab<n\leqslant M}\chi(n+ab)\right|\leqslant \frac{c}{4} N^{1-1/r}p^{(r+1)/4r^2}   \mathcal{T}_0^{1/2r},
$$
and 
$$
\left|\sum_{M+N-ab<n\leqslant M+N}\chi(n+ab)\right|\leqslant \frac{c}{4} N^{1-1/r}p^{(r+1)/4r^2}   \mathcal{T}_0^{1/2r},
$$
which combined with the above implies that 
$$
\left|\sum_{M<n\leqslant M+N}\chi(n)-\sum_{M<n\leqslant M+N}\chi(n+ab) \right|
\leqslant \frac{c}{2} N^{1-1/r}p^{(r+1)/4r^2}  \mathcal{T}_0^{1/2r}.
$$
 
The main difference with the method of \cite{Burgessrefine} comes from our choice of the subset used to average. We sum $w(a)$ over $a\leqslant A$ and $1\leqslant b \leqslant B$ and obtain 
\begin{equation}
\label{eq:Win}
\left|\sum_{M<n\leqslant M+N}\chi(n)\right|\leqslant \frac{S}{B||w||_1}+\frac{c}{2} N^{1-1/r}p^{(r+1)/4r^2}  \mathcal{T}_0^{1/2r},
\end{equation}
where
\begin{equation}
\label{eq:WW123}
S:=\sum_{M<n\leqslant M+N}\sum_{a \leqslant A}w(a)\left|\sum_{1\leqslant b \leqslant B}\chi(n+ab) \right|.
\end{equation}
By multiplying the innermost summation in (\ref{eq:WW123}) by $\chi(a^{-1})$ and 
collecting the values of $na^{-1} (\bmod\, q)$, we arrive at 
\begin{equation}
\label{eq}
S= \sum_{1\leqslant u \leqslant p}T(u) \left|\sum_{1\leqslant b \leqslant B}\chi(u+b) \right|,
\end{equation}
where 
$$T(u):=\sum_{a \leqslant A}w(a)\sum_{\scriptstyle M<n\leqslant M+N\atop\scriptstyle  n \equiv  u a  (\bmod{p}) }1.$$
Proceeding as in \cite{Burgessrefine}, the H\"older inequality gives 
$$
S^{2r}\leqslant \left(\sum_{u=1}^{p}T(u) \right)^{2r-2}\left(\sum_{u=1}^{p}T(u)^2 \right)\left(\sum_{u=1}^{p}\left|\sum_{1\leqslant b \leqslant B}\chi(u+b) \right|^{2r}\right).
$$

By Lemma~\ref{moments}, we see that 
\begin{equation}
\label{eq:weil1}
\sum_{u=1}^{p}\left|\sum_{1\leqslant b \leqslant B}\chi(u+b) \right|^{2r}\leqslant (2r)^r B^{r}p + 2rB^{2r}p^{1/2}.
\end{equation}
We trivially have
\begin{equation}
\label{eq:lin1}
\sum_{u=1}^{p}T(u)=\sum_{M<n\leqslant M+N}\sum_{a \leqslant A}w(a)=N\sum_{a \leqslant A}w(a)=N||w||_1.
\end{equation} 
Furthermore, we have 
$$\sum_{u=1}^{p}T(u)^2= T_w(M,N,A)$$
where $T_w(M,N,A)$ is as in  Lemma \ref{abcd}. We choose $w $ to minimize $\mathcal{T}_0(A)$.
By Lemma~\ref{abcd} and the hypothesis $N\leqslant p^{1/2+1/4r}$,  we have  
\begin{equation}
\label{moment2s}
\sum_{u=1}^{p}T(u)^2\ll ||w||_1^2 \big(1+  N\mathcal{T}_0/A\big) \ll ||w||_1^2 N\mathcal{T}_0/A. 
\end{equation}  

From \eqref{eq:weil1}, \eqref{eq:lin1} and~\eqref{moment2s}, we deduce
$$
S^{2r}\ll (2r)^{2r-1} p^{3/2}N^{2r-1} \mathcal{T}_0||w||_1^{2r}/A,
$$
and hence by our choice of parameters, it follows that there exists an absolute constant $c'$ such that 
$$ \frac{S}{ ||w||_1 B} \leqslant c'N^{1-1/r}p^{(r+1)/4r^2} \mathcal{T}_0^{ {1}/{2r}}  .$$

 Choosing $c=2c'$ and inserting in \eqref{eq:Win} implies \eqref{induction} 
\begin{equation*}
\left|\sum_{M<n\leqslant M+N}\chi(n)\right|	 \leqslant  cN^{1-1/r}p^{(r+1)/4r^2}\mathcal{T}_0^{1/2r}
\end{equation*} which concludes the proof by induction.

\section{Previous results and approaches concerning non vanishing of theta functions}\label{Def} 
\hspace{\parindent}   
In order to prove that $\theta (x,\chi)\neq 0$ for many of the $\chi\in X_p^+$, one may proceed as usual and study the asymptotic behavior of the moments of these theta values 
$$S_{2k}^+(p) :=\sum_{\chi\in X_p^+}\vert\theta (x,\chi)\vert^{2k}  \hspace{10mm} (k > 0).$$ 

Using the computation of the second and fourth moment, it was proved in \cite{LM} that $\theta (1,\chi)\neq 0$ for at least $\gg p/\log p$ of the $\chi\in X_p^+$. Lower bounds of good expected order for the moments were obtained in \cite{thetalow} as well as nearly optimal upper bounds conditionally on GRH in \cite{thetaupp}. This can be related to recent results of~\cite{HarperMaks}, where the authors obtain the asymptotic behavior of moments of Steinhaus random multiplicative function (a multiplicative random variable whose values at prime integers are uniformly distributed on the complex unit circle). This can reasonably be viewed as a random model for $\theta(x,\chi)$. Indeed, the rapidly decaying factor ${\rm e}^{-\pi n^2/q}$ is mostly equivalent to restrict the sum over integers $n \leqslant n_0(q)$ for some $n_0(q) \approx \sqrt{q}$ and the averaging behavior of $\chi(n)$ with $n\ll q^{1/2}$ is essentially similar to that of a Steinhaus random multiplicative function. As noticed by Harper, Nikeghbali and Radziwill in \cite{HarperMaks}, an asymptotic formula for the first absolute moment $S_{1}^+(p)$
 would probably imply the existence of a positive proportion of characters such that $\theta(x,\chi) \neq 0$. Though, quite surprisingly,  Harper proved recently both in the random and deterministic case that the first moment exhibits unexpectedly more than square-root cancellation \cite{Harpelow,Harperhigh}
\begin{equation}\label{squareroot}\frac{1}{p-1} \sum_{\chi \neq \chi_0} \left\vert \sum_{n\leqslant N} \chi(n)\right\vert \ll \frac{\sqrt{N}}{\min\left\{(\log \log L)^{1/4}, (\log \log \log q)^{1/4}\right\}}\end{equation} where $L=\min\left\{N,q/N\right\}$. Harper's result shows that this approach would, in any case, fail to provide the existence of a positive proportion of ``good" characters. In the next section, we adapt another approach in order to improve existing results. Precisely, we introduce mollifiers chosen as suitable weighted Dirichlet polynomials which reduce the problem to  the minimization problems considered in Section \ref{Intro}.

Moreover, in section \ref{firstmomentcarac}, we state and prove a lower bound for the first moment \eqref{squareroot}.

\section{Proof of Theorem \ref{mainth}}\label{ProofTh}
\hspace{\parindent} 
For any even character $\chi\in X_p^+$, let us define 
 \begin{equation}\label{weight} M(\chi)=\sum_{m\leqslant \sqrt{p/3} }w(m) \overline\chi(m),\end{equation} where $w(m)$ denote  some non-negative weights which will be fixed later.  We consider the first  mollified moment 
\begin{equation}\label{mollif}
M_1(p) :=\sum_{\chi\in X_p^+}M(\chi)\theta (x,\chi) .
\end{equation} 
Let us define $$M_0(p):=\#\big\{ \chi\in X_p^+,\quad \theta(x;\chi)\neq 0 \big\}.$$  By H\"{o}lder inequality, we have 
\begin{equation}\label{Holder} M_1(p) \leqslant M_2(p)^{1/2} M_4(p)^{1/4} M_0(p)^{1/4},  \end{equation}
with 
$$M_2(p):= \sum_{\chi\in X_p^+}\vert \theta (x,\chi)\vert^2,\qquad M_4(p):=\sum_{\chi\in X_p^+}\vert M(\chi)\vert^4. $$
 In \cite{LM}, the authors computed an asymptotic formula for the fourth moment of theta functions showing that the main contribution comes from the solutions $m_1n_1 = m_2n_2$ and obtained a precise asymptotic formula for the related counting function 
$$\big|\big\{m_1n_1=m_2n_2, m_1^2+n_1^2+m_2^2+n_2^2 \leqslant x\big\}\big| \sim \tfrac{3}{ 8}x\log x.$$ If we want to improve this result, we have to reduce the effect of this logarithmic term. By (\ref{Holder}), the problem is related to a similar counting problem restricted to a subset of integers supported by the weight $w$. Precisely, from~\eqref{Holder}, we get the following lower bound.
\begin{lemma}\label{moment1} For large prime $p$ and any sequence $w\in [0,+\infty[^{[\sqrt{p/3}]}$, we have
\begin{equation*} M_0(p)  \gg \left(\sum_{n\leqslant \sqrt{p/3}}w(n)\right)^{4} \mathcal{E}\big(\sqrt{p/3},w\big)^{-1}\end{equation*} where $  \mathcal{E}(N,w)
$ is defined by \eqref{weightedenergy}. In particular, we have 
\begin{equation}\label{lowM0} M_0(p) \gg \frac{p}{\mathcal{E}(\sqrt{p/3})}.  \end{equation}
\end{lemma} 

\begin{proof}
Let us recall the classical orthogonality relations for the subgroup of Dirichlet even characters $X_p^+$ 
$$\sum_{\chi\in X_p^+}\chi (m)\overline\chi (n)
=\begin{cases}\tfrac12
(p-1) & \textrm{if } m\equiv\pm n\bmod p \textrm{ and } \gcd (m,p)=1,\\
0 & \textrm{otherwise}. \end{cases}$$
Thus we have  
\begin{align*}
M_1(p) &=  \sum_{\chi\in X_p^+} \sum_{m\leqslant \sqrt{p/3}} \overline\chi (m)w(m) \sum_{n\geqslant 1}\chi(n){\rm e}^{-\pi n^2x/p} \nonumber \\
&\geqslant  \frac{p-1}{2}\sum_{m\leqslant \sqrt{p/3}}w(m) {\rm e}^{-\pi m^2x/p}. \nonumber  
\end{align*} 
We deduce that
\begin{equation}\label{lowM1}M_1(p) \gg   p \sum_{m\leqslant \sqrt{p/3}} w(m) = p||w||_1 .
\end{equation}
In the same way, we have 
\begin{align*}M_4(p)&=
\sum_{\scriptstyle  m_1,m_2, n_1,n_2  \leqslant \sqrt{p/3} }\!\!\!\!\!\!\!w(m_1)w(m_2)w(n_1)w(n_2)
\sum_{\chi\in X_p^+}\chi (m_2n_2)\overline\chi (m_1n_1)
\cr 
&=\tfrac12
(p-1)
\sum_{\scriptstyle  m_1,m_2, n_1,n_2  \leqslant \sqrt{p/3} \atop \scriptstyle m_1n_1=m_2n_2}\!\!\!\!\!\!\!w(m_1)w(m_2)w(n_1)w(n_2)
\cr&=\tfrac12
(p-1)\mathcal{E}(\sqrt{p /3},w ),\end{align*}
and 
\begin{align*}
M_2(p)&= \tfrac12
(p-1)  \sum_{ \scriptstyle  n_1,n_2 \geqslant 1\atop
\scriptstyle 
n_1\equiv \pm n_2\bmod p }      {\rm e}^{- \pi (n_1^2+n_2^2)x/p} \ll p^{3/2}.
\end{align*}
Reporting these estimates in \eqref{Holder}, we finish the proof of Lemma \ref{moment1}.
\end{proof}  

Lemma \ref{moment1}   combined with Theorem \ref{multh} finishes  the proof of Theorem \ref{mainth}.

\section{Proof of Theorem \ref{firstmomentcarac}}
\hspace{\parindent} We adopt similar techniques as the ones used in Section \ref{ProofTh}. In a similar way as in \eqref{weight} we define $$M_{\chi}(N)=\sum_{m\leqslant N}w(m)\overline{\chi}(m),$$ 
where $w=(w(n))_{1\leqslant n\leqslant N}\in [0,+\infty[^N$. We introduce the parameters $\alpha=\frac{r}{4-2r}$ and $\beta=\frac{8-6r}{8-4r}$ such that $\alpha+\beta=1$. We further set $p=4-2r$ and $q=\frac{8-4r}{4-3r}$ which verify $\frac{1}{p}+\frac{1}{q}+\frac{1}{4}$=1.
 Writing $S_{\chi}(N)=S_{\chi}(N)^{\alpha}S_{\chi}(N)^{\beta}$ and applying H\"{o}lder's inequality, we have 
\begin{equation}\label{momentsr} \frac{1}{p-1}\left\vert \sum_{\chi \neq \chi_0}S_{\chi}(N)M_{\chi}(N)\right\vert \leqslant \mathfrak{S}_r(N)^{\frac{1}{4-2r}} \mathfrak{S}_2(N)^{\frac{4-3r}{8-4r}} \mathfrak{M}_4(N)^{1/4},  \end{equation}
where, for any $k>0$, we have
$$\mathfrak{S}_k(N):= \frac{1}{p-1} \sum_{\chi \neq \chi_0} \vert S_{\chi}(N) \vert^k, \qquad \mathfrak{M}_4(N):=\frac{1}{p-1} \sum_{\chi \neq \chi_0} \vert M_{\chi}(N) \vert^4. $$
Using orthogonality relations, it is easy to see that $\mathfrak{S}_2(N) \ll N$. In the same manner as in Section \ref{ProofTh}, the left hand side of \eqref{momentsr} is bounded from below by $||w||_1$. Similarly, we have $\mathfrak{M}_4(N) \ll \mathcal{E}(N,w)$. Combining together these inequalities, we deduce
$$\mathfrak{S}_r(N)^{\frac{4}{4-2r}} \gg \frac{||w||_1^4}{N^{\frac{4-3r}{2-r}}\mathcal{E}(N,w)}.$$ Hence we get
$$\mathfrak{S}_r(N) \gg \frac{N^{r/2}}{\mathcal{E}(N)^{1-r/2}}.$$


\section{Concluding remarks}\label{combiproblems}
\hspace{\parindent} 
 Under the additional restriction $w(m)\in \left\{0,1\right\}$, our first problem considered in Section \ref{Intro} is equivalent to the construction of a set $\mathcal{B}\subset [1,N]$ of high density such that the associated GCD sum is small. In the present article we showed that the set $\mathcal{B}$ of integers having exactly $k$ prime factors with~$k=\kappa^* \log_2 N$ and $\kappa^* \approx 0.48154$ (which is a set of density $(\log N)^{-\delta_0+o(1)}$ with $\delta_0 \approx 0.16656$) verifies 
$$\sum_{m_1,m_2 \in \mathcal{B} }\frac{(m_1,m_2)}{m_1+m_2} \ll \vert \mathcal{B}\vert (\log\vert \mathcal{B}\vert)^{o(1)}$$ or in another words the multiplicative energy verifies
$$E_{\times}(N,\mathcal{B}) \ll N \vert \mathcal{B}\vert (\log N)^{o(1)}. $$ Theorem \ref{gcdsumth} shows that it is essentially the densest set having this property. In the symmetric case the question becomes: what is the maximal $0 < \beta    < 1$ (in terms of $N$) such that there exists a set $\mathcal{B}\subset \left[1,N\right]$ of density $\beta$ verifying the upper bound $ E_{\times}(\mathcal{B},\mathcal{B})\ll \vert \mathcal{B}\vert^2 (\log N)^{o(1)}. $ In Theorem~\ref{multh}, we proved that the set of integers having exactly $k$ prime factors with $k=\big[\frac{\log_2 N}{\log 4}\big]$ gives the optimal density $(\log N)^{-\delta/2+o(1)}$.

\section*{Acknowledgements}\hspace{\parindent} 
The first author gratefully acknowledges comments from G\'{e}rald Tenenbaum. The second author would like to thank St\'{e}phane Louboutin for valuables remarks as well as Igor Shparlinski for pointing him out the reference~\cite{Burgessrefine} after a first version of the draft was released. 
The authors would like to thank Kannan Soundararajan for drawing their attention to~\cite{BS16}.
The second author acknowledges support of the Austrian Science Fund (FWF), START-project Y-901 ``Probabilistic methods in analysis and number theory" headed by Christoph Aistleitner.

\Addresses

\end{document}